\documentclass[12pt]{article}

\usepackage{cmap}  
\usepackage[T2A]{fontenc}
\usepackage[utf8]{inputenc}
\usepackage[english]{babel}
\usepackage[usenames]{color}
\usepackage{amsmath,amssymb,amsthm}
\usepackage{arcs}
\usepackage{epsfig}
\usepackage{filecontents}
\usepackage{graphicx}
\usepackage[pdftex,colorlinks=true,linkcolor=blue,urlcolor=red,unicode=true,hyperfootnotes=false,bookmarksnumbered]{hyperref}
\usepackage{ifthen}
\usepackage{import}
\usepackage{indentfirst}
\usepackage{mathtools}
\usepackage{cancel}
\usepackage{xcolor}

\usepackage[margin=1in, total={8.5in, 11in}]{geometry}

\renewcommand{\le}{\leqslant}
\renewcommand{\ge}{\geqslant}

\newcommand{\p}{\mathcal{P}}

\newcommand{\ff}{\mathcal{F}}
\newcommand{\s}{\mathcal{S}}
\newcommand{\g}{\mathcal{G}}
\newcommand{\aaa}{\mathcal{A}}

\newcommand{\hh}{\partial}
\newcommand{\h}{\mathcal H}

\newcommand{\E}{\mathsf{E}}

\newcommand{\Prb}{\mathsf{P}}

\newcommand{\T}{\mathcal{T}}
\newcommand{\W}{\mathcal{W}}

\newtheorem{thm}{Theorem}

\newtheorem{lem}[thm]{Lemma}

\newtheorem{prop}[thm]{Proposition}

\newtheorem{obs}[thm]{Observation}

\title{Spread approximations for forbidden intersections problems}
\author{Andrey Kupavskii\footnote{Moscow Institute of Physics and Technology and St. Petersburg State University, Russia; Email: {\tt kupavskii@ya.ru}.}, Dmitrii Zakharov\footnote{
Department of Mathematics, Massachusetts Institute of Technology, Cambridge, MA 02139, USA.  
Email: {\tt zakharov2k@gmail.com}}}
 
\date{}

\begin{document}
\maketitle
\begin{abstract}
  We develop a new approach to approximate families of sets, complementing the existing `$\Delta$-system method' and `junta approximations method'. The approach, which we refer to as `spread approximations method', is based on the notion of $r$-spread families and builds on the recent breakthrough result of Alweiss, Lovett, Wu and Zhang for the Erd\H os--Rado `Sunflower Conjecture'. Our approach can work in a variety of sparse settings.

  To demonstrate the versatility and strength of the approach, we present several of its applications to forbidden intersection problems, including bounds
  on the size of regular intersecting families, the resolution of the Erd\H os--S\'os problem for sets in a new range and, most notably, the resolution of the $t$-intersection and Erd\H os--S\'os problems for permutations in a new range. Specifically, we show that any collection of permutations of an $n$-element set with no two permutations intersecting in at most (exactly) $t-1$ elements has size at most $(n-t)!$, provided $t\le n^{1-\epsilon}$ ($t \le n^{\frac{1}{3}-\epsilon}$) for an arbitrary $\epsilon>0$ and $n>n_0(\epsilon)$. Previous results for these problems only dealt with the case of fixed $t$. The proof follows  the structure vs. randomness philosophy, which proved to be very efficient in proving results throughout mathematics and computer science.

\end{abstract}

\section{Introduction}
We start the discussion not with the spread approximations method itself, but with some of its implications to questions in extremal set theory that deal with forbidden intersections. The method, as well as the structural results that it gives, will be discussed in the latter sections. 

We use standard notation $[n]$ for the set $\{1,\ldots, n\}$, $2^X$ for the power set of $X$ and ${X\choose k}$ for the set of all $k$-element subsets of $X$. A collection, or a family, $\ff$ of sets is {\it $t$-intersecting} if $|A\cap B|\ge t$ for any $A,B\in \ff$, and is called {\it intersecting} if it is $1$-intersecting.

One of the main themes of extremal combinatorics, with numerous applications in other branches of mathematics and computer science, is the study of hypergraphs or families of sets that avoid certain substructures.  Among the most influential results of this kind is the Erd\H os--Ko--Rado theorem \cite{EKR} that determines the size of the largest intersecting family of $k$-element subsets of $[n]$. It grew into a separate domain of extremal combinatorics that studies collections of various mathematical objects with restrictions on their intersections.

The theorem of Ahlswede and Khachatrian \cite{AK}, extending earlier results  by Frankl and F\"uredi \cite{F1, FF3}, determines  the largest $t$-intersecting families of $k$-element subsets in $[n]$ \emph{ for all $n,k,t$}. This result played an important role in the result of Dinur and Safra \cite{DS} that showed that vertex cover is hard to approximate via an efficient algorithm. (To mention some more connections between intersection problems and Computer Science,  Frankl and Wilson \cite{FW} used their forbidden intersections result in order to provide  explicit constructions of Ramsey graphs; Razborov and Vereshchagin \cite{RV} used cross-intersecting families in the context of satisfiability of DNF.)

A harder question, due to Erd\H os and S\'os  from 1971 (cf. \cite{E74}), still remains unsolved in general: what is the largest family $\ff\subset{[n]\choose k}$ that {\it avoids intersection $t-1$}, i.e. such that $|A\cap B|\ne t-1$ for any $A,B\in \ff$? This was answered for $n>n_0(k)$ and $2t<k$ by Frankl and F\"uredi \cite{FF1} using the so-called $\Delta$-system method. The variant of the $\Delta$-system method they use requires $n_0(k)$ to be at least doubly-exponential in $k$.

One of the big advancements in the area in the recent years came from the works of Keller and Lifshitz \cite{KL} and  Ellis, Keller and Lifshitz \cite{EKL}, who greatly developed the so-called junta method for dealing with forbidden intersections and similar extremal problems. In particular, their results gave the solution to the Erd\H os--S\'os question for almost any value of $n,k$, provided that $n>n_0(t)$, for some unspecified  function $n_0(t)$.  We also mention further developments by Keevash, Lifshitz, Long, and Minzer \cite{KLLM, KLLM2}.

We also should mention an important result due to Frankl and Wilson \cite{FW}, and later Frankl and R\"odl \cite{FR} that is situated on the other part of the spectrum of possible values of parameters. They gave strong upper bounds for the size of the largest family in ${[n]\choose k}$ that avoids intersection $t-1$, where both $k$ and $t$ were {\it linear} in $n$. Their results had several important consequences for problems in Discrete Geometry, such as Borsuk's problem, the chromatic number of the space and for more general Euclidean Ramsey Theory type questions (cf. \cite{FR2}, \cite{Rai1}). 

One of the consequences of our spread approximations method is the resolution of the Erd\H os--S\'os problem for $t$ that is allowed to grow {\it polynomially} with $n$.

\begin{thm}\label{thmintro1}
  Fix some $\alpha>1$ and $\beta\in (0,0.5)$ satisfying $\alpha>1+2\beta$ and let $k_0>0$ be a sufficiently large integer.  For any $k\ge k_0$ put $n = \lceil k^{\alpha}\rceil $ and  $t = \lceil k^{\beta}\rceil $. If $\ff\subset{[n]\choose k}$ avoids intersection $t-1$ then $|\ff|\le {n-t\choose k-t}$.
\end{thm}

Another type of structures that are very well-studied in the context of forbidden intersections are permutations. We denote by $\Sigma_n$ the family of all permutations $[n]\to [n]$. For two permutations $\sigma, \pi: [n]\to [n]$, the size of their intersection, denoted $|\sigma\cap \pi|$, is the number of $x\in [n]$ such that $\sigma(x) = \pi(x)$. It is convenient to identify a permutation $\sigma \in \Sigma_n$ with its graph, i.e. an $n$-element subset of $[n]^2$ consisting of pairs $(x, \pi(x))$, $x \in [n]$. Then the above notion of intersection of permutations translates to the usual set intersection of the graphs of these permutations.
All definitions (intersecting, avoiding intersection $t-1$ etc.) and questions discussed above carry over to the families of permutations.

Deza and Frankl \cite{DeF} showed that any intersecting family of permutations in $\Sigma_n$ has size at most $(n-1)!$.
This is the size of a family of permutations that map some element $x$ to another element $y$ or, in other words, the family of permutations whose graphs contain a fixed element $(x, y) \in [n]^2$.
Unlike in the case of the Erd\H os--Ko--Rado theorem for intersecting families of sets, even the uniqueness of the described extremal example was hard to obtain and was proved independently by Cameron and Ku \cite{CK} and Larose and Malvenuto \cite{LM}.

More generally, for $t\ge 1$, fix distinct elements $x_1, \ldots, x_t \in [n]$ and distinct $y_1, \ldots, y_t \in [n]$. The family of permutations $\sigma \in \Sigma_n$ such that $\sigma(x_i) = y_i$ for every $i=1, \ldots, t$ is $t$-intersecting and has cardinality $(n-t)!$. We call such families, and all their subfamilies, {\it trivial $t$-intersecting families}. We say that a $t$-intersecting family of permutations $[n]\to [n]$ is {\it non-trivial} if it is not contained in any trivial $t$-intersecting family.
Deza and Frankl proposed the following conjecture: the largest $t$-intersecting subfamily of $\Sigma_n$ has size $(n-t)!$ for $n>n_0(t)$. The conjecture was proved in a paper by Ellis, Friedgut and Pilpel \cite{EFP}. Recently, this was extended by Ellis and Lifshitz \cite{EL} to the Erd\H os--S\'os setting: for $n>n_0(t)$ any subfamily of $\Sigma_n$ that avoids intersection $t-1$ has size at most $(n-t)!$. We note that the proofs of the last two results are hard and require representation theory of symmetric groups together with Fourier analysis. We also note that there is a weak analogue of the Frankl and Wilson result \cite{FW} for permutations, which is due to Keevash and Long \cite{KeeL}.

Arguably the most impressive application of the spread approximations method that we found so far is the following result that greatly extends the aforementioned results of Ellis, Friedgut and Pilpel and of Ellis and Lifshitz. The proof is also much simpler and avoids the use of heavy machinery of the previous authors.

\begin{thm}\label{thmintro2}
  Let $n,t>0$ be integers and $\p\subset \Sigma_n$ be a family of permutations.
  \begin{itemize}
    \item[(i)] If $\p$ is $t$-intersecting then $|\p|\le (n-t)!$, provided $n> 2^{22} t\log_2^2 n$. Moreover, if $\p$ is non-trivial then it has size at most $\frac 23(n-t)!.$
    \item[(ii)] If $\p$ avoids intersection $t-1$ then $|\p|\le (n-t)!$, provided $n>2^{25}t^2\log_2^2 n$ and $n>2^{10}t^3 \log_2n$.
  \end{itemize}
\end{thm}

In fact, Theorem \ref{thmintro2} is a particular case of two more general results that we formulate and prove in Sections~\ref{sec3} (analogue of (i)) and~\ref{sec4} (analogue of (ii)). Roughly speaking, that results state that for a sufficiently `nice' family $\aaa \subset {[n] \choose k}$, the maximal size of a $t$-intersecting  (or $(t-1)$-avoiding) subfamily $\ff \subset \aaa$ is achieved on a family of the form $\aaa(T)$ for $T \in {[n] \choose t}$, that is, the family of sets $A \in \aaa$ containing a fixed set $T$ of size $t$. The dependencies of the parameters and the conditions are stricter for the $(t-1)$-avoidance.  
Theorem~\ref{thmintro2} can be derived as a corollary of these results for  $\aaa = \Sigma_n \subset {[n]^2 \choose n}$.

We remark that even though by Theorem \ref{thmintro2} the maximal size of a non-trivial $t$-intersecting family is less than $(n-t)!$ by a constant factor, one does not have an analogous statement for non-trivial families which avoid intersection $t-1$. Indeed, consider the following Hilton-Milner-type family: let $\sigma = (1 2 \ldots t)(t+1)\ldots (n)$ be a permutation that acts cyclically on the first $t$ elements and fixes the remaining $n-t$, and let $\p' \subset \Sigma_n$ be the family of permutations $\pi$ such that $\pi(x) = x$ for $x \in [t]$ and $|\pi \cap \sigma| \neq t-1$. Then the family $\p = \p' \cup \{\sigma\}$ is non-trivial and avoids intersection $t-1$. A simple computation shows that if $t \rightarrow \infty$ then $|\p| = (1 - o(1))(n-t)!$.

We note that the junta approximation results of \cite{EFP} and \cite{EL}, as well as the results for sets from \cite{KL} and \cite{EKL} came together with stability results that characterized the approximate structure of families that are close to extremal. Our approach also gives such  structural results, moreover, it allows finding structure in families of much smaller size than that suitable for the use of the junta method. Our general results in Section~\ref{sec3} and~\ref{sec4} allow for similar results in a sparse setting that have no analogues via the junta method. And, of course, this opens the way to exploring other Tur\'an-type problems for subfamilies of `nice' families, such as the family of all permutations.

Lastly, let us return to intersecting families of $k$-sets. On the one extreme, we have the  Erd\H os--Ko--Rado theorem, which is tight for families of sets containing a fixed element. On the other extreme, we have intersecting families that are regular (all elements of the ground set are contained in the same number of sets) or even symmetric (its group of automorphisms is transitive). Note that if a family is symmetric, then it is regular, but not vice versa. We show the following result that answers a question of  Narayanan.

\begin{thm}\label{thmintro3} Let $n,k,q\ge 1$ be integers such that $n> 2^{13} k\log_2(2k)$ and $n\ge 4q k$. Let  $\ff\subset {[n]\choose k}$ be a regular intersecting family. Then $|\ff|\le 2^{-q}{n\choose k}$.
\end{thm}

Previously, a similar bound was shown by Ellis, Kalai and Narayanan \cite{EKN} for symmetric intersecting families in ${[n]\choose k}$. The best known upper bound for regular families was due to Ihringer and Kupavskii \cite{IK} and stated that any regular intersecting family $\ff$ in ${[n]\choose k}$ satisfies $|\ff|\le {n-r\choose k-r}$ for $n>c(r)k$, i.e., that $\ff$ is polynomially small in $k/n$ for sufficiently large $n$. The lower bound construction, presented in \cite{IK}, implies that there is an $\ff$ as in Theorem~\ref{thmintro3} with roughly $|\ff|\ge q^{-q}{n\choose k}$, which shows that the bound in Theorem~\ref{thmintro3} is not so far from the truth.

Juntas and Junta approximations are an important part of several areas of Mathematics and Computer Science, such as Analysis of Boolean Functions, Property Testing, Hardness of Approximation, Machine Learning etc., see, e.g., \cite{Fish}. Our spread approximation approach provides another way of decomposing Boolean functions. We believe that it can complement existing tools and give new insights into the structure of different classes of Boolean functions.

The rest of the paper is structured as follows. 
The next section is devoted to $r$-spread measures and the result of Alweiss, Lowett, Wu, and Zhang \cite{Alw}. In Section~\ref{sec3}, we present the easier variant of the  spread approximation approach that is suitable for dealing with intersecting and $t$-intersecting structures and prove Theorem~\ref{thmintro3}. In Section~\ref{sec4} we present the variant of spread approximations suitable for the Erd\H os--S\'os problem. In Section~\ref{sec5} we give concrete scenarios in which our general results are applicable. One of the applications of our results is an Ahlswerde--Khachatrian type result for designs. In Section~\ref{sec6}, we discuss some other implications of the method and directions for future investigation.

\paragraph{Acknowledgements.} We thank an anonymous referee for helpful comments and for suggesting a simplification of the proof of Lemma \ref{ess}.

\section{Spread families}\label{sec2}
Given a probability measure $\mu:2^{[n]}\to [0,1]$, we say that $\mu$ is {\it $r$-spread} for some $r\ge 1$ if $\mu(\{F: X\subset F\})\le r^{-|X|}$ for any $X\subset [n]$. We note that there is a natural measure $\mu_\ff$ that corresponds to a family $\ff\subset 2^{[n]}:$ $\mu_\ff(F) = \frac 1{|\ff|}$ if $F\in \ff$ and $0$ otherwise. We will mostly work with such measures, and will say that the family is $r$-spread if the corresponding measure is $r$-spread. We denote by $|\mu|$ the expected size of a set sampled from $\mu$: $|\mu|:=\sum_{S\subset 2^{[n]}} |S|\mu(S)$. Note that $|\mu|\le n$ for any $\mu$. The {\it support} ${\rm supp}(\mu)$ is the set of all $S\subset [n]$ such that $\mu(S)>0$.

We say that $W$ is a {\it $p$-random subset} of $[n]$ if each element of $[n]$ is included in $W$ with probability $p$ and independently of others.
The following statement is a variant due to Tao \cite{Tao} of the breakthrough result that was proved by Alweiss, Lowett, Wu and Zhang \cite{Alw}.

\begin{thm}[\cite{Alw}, a sharpening due to \cite{Tao}]\label{thmtao}
  If for some $n,r\ge 1$ a measure $\mu: 2^{[n]}\to [0,1]$ is $r$-spread and $W$ is an $(m\delta)$-random subset of $[n]$, then $$\Pr[\exists F\in {\rm supp}(\mu)\ :\ F\subset W]\ge 1-\Big(\frac 5{\log_2(r\delta)} \Big)^m|\mu|.$$
\end{thm}

This theorem has an important implication for the classical problem of Erd\H os and Rado \cite{ER} on sunflowers. Recall that an {\it $\ell$-sunflower} is a collection of $\ell$ sets $F_1,\ldots, F_\ell$ such that for any $i\ne j$ we have $F_i\cap F_j = \cap_{i=1}^\ell F_i$. (In particular, $\ell$ pairwise disjoint sets form an $\ell$-sunflower.) The set $\cap F_i$ is called the {\it core}.  Erd\H os and Rado showed that if $\ff$ is a family of $k$-element sets such that  $|\ff|> k!(\ell-1)^{k}$, then it contains an $\ell$-sunflower. They conjectured that the same holds if $|\ff|>C^k$ for some $C  =C(\ell).$ Theorem~\ref{thmtao} almost immediately gives the first major improvement of the bound due to  Erd\H os and Rado, implying that any family $\ff$ of $k$-sets with \begin{equation}\label{eqalweiss}|\ff|> \big(C\ell\log_2 (k\ell)\big)^k\end{equation} contains an $\ell$-sunflower, where $C$ is an absolute constant and can be taken to be $2^{10}$. The idea is as follows. First, given a large family $\ff$, restrict to a ``subfamily'' of the form $$\ff(S):=\{A\setminus S: A\in\ff, S\subset A\}$$
that is sufficiently spread. Second, given a spread family $\ff(S)$, randomly color the set $[n]\setminus S$ into $\ell$ colors. Applying Theorem~\ref{thmtao}, with positive probability, we simultaneously have a monochromatic set from $\ff(S)$ in each of the colors. Together, they form an $\ell$-matching in $\ff(S)$ and thus an $\ell$-sunflower with core $S$ in $\ff$. We note that a more efficient version of this argument due to \cite{BCW} yields a slightly better bound in (\ref{eqalweiss}). For clarity of exposition in what follows, we define $\ff[S]$ to be
$$\ff[S]:=\{A: A\in\ff, S\subset A\}.$$
Note the difference between $\ff[S]$ and $\ff(S)$. In particular, we have $\ff[S]\subset \ff,$ which is generally not true for $\ff(S).$

Theorem~\ref{thmtao} gives something more robust than just a sunflower (actually, the authors of \cite{Alw} use terminology {\it robust sunflowers}): with large probability it guarantees sunflowers that `respect' a random partition. Talagrand \cite{Tal} observed using duality that the property of a family being $r$-spread is roughly equivalent to having a fractional cover that has small $r^{-1}$-biased measure. Frankston, Kahn, Narayanan and Park \cite{FKNP} used this and a slight modification of the proof from \cite{Alw} to show a general result that relates thresholds and fractional expectation-thresholds for monotone properties. \\

For us, it will be handy to work with the following modified version of spreadness, which we define only for (unweighted) uniform families. For $\tau> 1$ we say that a family $\ff\subset {[n]\choose k}$ is {\it $\tau$-homogeneous} if for any $A$ of size $a\le k$ we have $|\ff(A)|\le \tau^a \frac{{n-a\choose k-a}}{{n\choose k}}|\ff|$. Note that, by simple averaging, for any family $\ff\subset {[n]\choose k}$ there always exists $A$ of size $a$ such that $|\ff(A)|\ge \frac{{n-a\choose k-a}}{{n\choose k}}|\ff|$. Thus, if we ask for the $1$-homogeneous property, then each $\ell$-set should be contained in the same number of sets from the family, and this should be true for each $\ell\le k$. It should be clear that this is only possible if $\ff = {[n]\choose k}$.
\begin{obs}\label{obs1}
If $\ff\subset {[n]\choose k}$ is $\tau$-homogeneous then it is $\frac {n}{\tau k}$-spread.
\end{obs}
\begin{proof}
For any set $A\subset [n]$ of size $a$ we have $|\ff(A)|\le \tau^a \frac{{n-a\choose k-a}}{{n\choose k}}|\ff|\le \tau^a \big(\frac kn\big)^a|\ff|,$ showing that $\ff$ is $\frac{n}{\tau k}$-spread.
\end{proof}
\begin{obs}\label{obs2} Given $\tau>1$ and a family $\ff\subset {[n]\choose k}$, let $X$ be a maximal set that satisfies $|\ff(X)|\ge\tau^{|X|} \frac{{n-|X|\choose k-|X|}}{{n\choose k}}|\ff|$. Then $\ff(X)$ is $\tau$-homogeneous as a family in ${[n]\setminus X\choose k-|X|}$.
\end{obs}
\begin{proof}
Indeed, for any $B\supsetneq X$ of size $b$ we have
$$|\ff(B)|\le \tau^b \frac{{n-b\choose k-b}}{{n\choose k}}|\ff|\le\tau^{b-|X|} \frac{{n-b\choose k-b}}{{n-|X|\choose k-|X|}}|\ff(X)|.$$
\end{proof}

An important new notion that we introduce is  a {\it relative} notion of homogeneity. In particular, it is central in the proof of our results for families of permutations. Fix an arbitrary family of sets $\mathcal A$ and a real number $\tau > 1$. A subfamily $\ff \subset \aaa$ is called {\it $(\aaa, \tau)$-homogeneous} (or {\it $\tau$-homogeneous relative to $\aaa$}) if for any set $S$ we have
$$
\frac{|\ff(S)|}{|\ff|} \le \tau^{|S|}\frac{|\aaa(S)|}{|\aaa|}.
$$
That is, if a family $\ff\subset {[n]\choose k}$ is $\tau$-homogeneous, then it is $\tau$-homogeneous relative to  ${[n]\choose k}.$ The following observation is virtually identical to Observation~\ref{obs2}, but we present the proof for clarity.
\begin{obs}\label{obs3} Given a family $\aaa$, $\tau>1$ and a family $\ff\subset \aaa$, let $X$ be a maximal set that satisfies $|\ff(X)|\ge\tau^{|X|} \frac{|\aaa(X)|}{|\aaa|}|\ff|$. Then $\ff(X)$ is $(\aaa(X),\tau)$-homogeneous.
\end{obs}
\begin{proof}
Indeed, for any $B\supsetneq X$ of size $b$ we have
$$|\ff(B)|\le \tau^b \frac{|\aaa(B)|}{|\aaa|}|\ff|\le\tau^{b-|X|} \frac{|\aaa(B)|}{|\aaa(X)|}|\ff(X)|.$$
\end{proof}

\section{Spread approximations for Ahlswede--Khachatrian}\label{sec3}
In this section, we present a simpler version of our spread approximation method, which also requires the weakest assumptions on the parameters.
We first deal with the easiest and most classical case of  $t$-intersecting families of sets in ${[n]\choose k}$. It easily extends to several cross-dependent families, but, in order to avoid technicalities, we postpone this  more general statement to Section~\ref{sec6}.

\begin{thm}\label{thm2}
  Let $n,k\ge 2$, $t\ge1$ be some integers. Consider a family $\ff\subset {[n]\choose k}$ that is  $t$-intersecting. Let $q, \tau \ge 1$ satisfy the following restrictions: $n> 2^{12}\tau k\log_2(2k)$, $n\ge 2q\tau k$ and $q \ge t$. Then there exist a $t$-intersecting  family $\mathcal S\subset {[n]\choose \le q}$ of sets of size at most $q$ and a family $\ff'$ such that the following holds.
  \begin{itemize}
    \item[(i)] For all $A\in \ff\setminus \ff'$ there is $B\in\s$ such that $B\subset A$;
    \item[(ii)] for any $B\in \s$ there is a family $\ff_B\subset \ff$ such that $\ff_B(B)$ is $\tau$-homogeneous;
    \item[(iii)] $|\ff'|\le \tau^{-q-1}{n\choose k}$.
  \end{itemize}
\end{thm}

\begin{proof}[Proof of Theorem~\ref{thm2}] The first guiding observation to make is that $\ff$ cannot be $\tau$-homogeneous. Indeed, arguing indirectly, randomly color $[n]$ into two colors. It is not difficult to verify that, under our assumptions, Observation~\ref{obs1} and Theorem~\ref{thmtao} imply that, with positive probability, we have a set from $\ff$ in each of the colors. But these sets are disjoint. We avoid calculations here, since we do them below in a more general scenario. This motivates the following procedure for $i=1,\ldots $ with $\ff^1:=\ff$.
\begin{enumerate}
    \item Find a maximal $S_i$ that  $|\ff^i(S_i)|\ge  \tau^{|S_i|} \frac{{n-|S_i|\choose k-|S_i|}}{{n\choose k}} |\ff^i|$.
    \item If $|S_i|> q$ or $\ff^i = \emptyset$ then stop. Otherwise, put $\ff^{i+1}:=\ff^i\setminus \ff^i[S_i]$.
\end{enumerate}

Note that $\ff^i(S_i)$ is $\tau$-homogeneous by Observation~\ref{obs2}.

Let $m$ be the step of the procedure for $\ff$ at which we stop. The family $\s$ is defined as follows: $\s:=\{S_1,\ldots, S_{m-1}\}$. Clearly, $|S_i|\le q$ for each $i\in [m-1]$. The family $\ff_{B}$ promised in (ii) is defined to be $\ff^i[S_i]$ for $B=S_i$. Next, note that if $\ff^m$ is non-empty, then $$|\ff^m|\le \tau^{-|S_m|} \frac{{n\choose k}}{{n-|S_m|\choose k-|S_m|}} |\ff^{m}(S_m)|\le  \tau^{-|S_m|} \frac{{n\choose k}}{{n-|S_m|\choose k-|S_m|}}{n-|S_m|\choose k-|S_m|} = \tau^{-|S_m|}{n\choose k}.$$
We put $\ff':=\ff^m$. Since either $|S_m|>q$ or $\ff' = \emptyset $, we have $|\ff'|\le \tau^{-q-1}{n\choose k}$.

We have verified all the properties required from $\ff$, $\ff'$ and $\s$, except for the following crucial property.
\begin{lem}\label{lemtint} The family $\mathcal S$ is $t$-intersecting.
\end{lem}
\begin{proof}

  Take any (not necessarily distinct) $S_i,S_j\in \mathcal S$  and assume that $|S_i\cap S_j|<t$. Recall that  $\g_i:=\ff^i(S_i)$ and $\g_j:=\ff^j(S_j)$ are both $\tau$-homogeneous. For a set $X\subset [n]$ and a family $\g\subset 2^{[n]}$ we use the following standard notation:
  $$\g(\bar X):=\{G\in \g: G\cap X = \emptyset\}.$$
  We think of $\g(\bar X)$ as of a subfamily of $2^{[n]\setminus X}.$ By Observation~\ref{obs1}, $\g_i$ and $\g_j$ are $\big(\frac {n}{\tau k}\big)$-spread. We use this in the second inequality below.
  $$|\g_j(\bar S_i)| \ge |\g_j|-\sum_{x\in S_i\setminus S_j} |\g_j[\{x\}]|\ge \Big(1-\frac {\tau k|S_i|}{n}\Big) |\g_j|\ge \frac 12|\g_j|.$$
  In the last inequality we used that $|S_i|\le q$ and that $n\ge 2\tau kq.$ The same is valid for $\g_i(\bar S_j)$. Note that both $\g_j':=\g_j(\bar S_i)$ and $\g_i':=\g_i(\bar S_j)$ are subfamilies of $2^{[n]\setminus (S_i\cup S_j)}.$ Because of the last displayed inequality and the trivial inclusion $\g_j'(Y)\subset \g_j(Y)$, valid for any $Y$, both $\g_i'$ and $\g_j'$ are $\big(\frac{n}{2\tau k}\big)$-spread. Thus, by the first assumption on $n$ in Theorem~\ref{thm2} $\g_i'$ and $\g_j'$ are $r$-spread with some $r > 2^{11}\log_2(2k)$.  

  We are about to apply Theorem~\ref{thmtao}. Let us put $m= \log_2(2k)$ and $\delta = (2\log_2(2k))^{-1}$. Note that $m\delta = \frac 12$ and $r\delta > 2^{10}$ by our choice of $r$.  Theorem~\ref{thmtao} implies that a $\frac{1}{2}$-random subset $W$ of $[n]\setminus (S_i\cup S_j)$ contains a set from $\g_j'$ with probability strictly bigger than
  $$1-\Big(\frac 5{\log_2 2^{10}}\Big)^{\log_2 2k} k = 1-2^{-\log_2 2k} k = \frac 12.$$

  Consider a random partition of $[n]\setminus (S_i\cup S_j)$ into $2$ parts $U_i,U_j$, including each element with probability $1/2$ in each of the parts. Then both $U_\ell$, $\ell\in \{i,j\}$, are distributed as $W$ above. Thus, the probability that there is $F_\ell \in \g_\ell'$ such that $F_\ell\subset U_\ell$ is strictly bigger than $\frac 12$. Using the union bound, we conclude that, with positive probability, it holds that there are such $F_\ell$, $F_\ell\subset U_\ell,$ for each  $\ell \in\{i,j\}$. Fix such choices of $U_\ell$ and $F_\ell$, $\ell \in \{i,j\}$. Then, on the one hand, both $F_i\cup S_i$ and $F_j\cup S_j$ belong to $\ff$ and, on the other hand, $|(F_i\cup S_i)\cap (F_j\cup S_j)| = |S_i\cap S_j|<t$, a contradiction with $\ff$ being $t$-intersecting. 
  \end{proof}
This concludes the proof of Theorem~\ref{thm2}.\end{proof}
We note that Theorem~\ref{thm2} can be formulated and proved for $r$-spread families instead of $\tau$-homogeneous families. Let us now derive Theorem~\ref{thmintro3}.

\begin{proof}[Proof of Theorem~\ref{thmintro3}]
The first thing to note is that  there are no regular intersecting families of $k$-sets if $n>k^2$. Indeed, the average degree of an element is $\frac kn|\ff|$, which should be the degree of {\it any} element due to regularity. At the same time, any set $F\in \ff$ intersects all other sets in $\ff$, and so one of its elements has degree at least $\frac 1k|\ff|$, which is bigger than the previous expression for $n>k^2$.

For $q\ge k$ the second restriction on $n$ in the theorem becomes $n\ge 4k^2$, and, as we have just seen, this is impossible. Thus, we suppose that $q<k$ in what follows.    Arguing indirectly, assume that $|\ff|> 2^{-q}{n\choose k}$. Apply Theorem~\ref{thm2} to $\ff$ with $\tau=2$, $t=1$ and $q$, obtaining the families $\s$ and $\ff'$, where $|\ff'|<\frac 12|\ff|$. Next, take any set $B\in S$ and note that all sets from $\ff\setminus \ff'$ intersect it. (This is because of (i) and the fact that $\s$ is intersecting.) Therefore, one of the elements in $B$ has degree at least
  $$\frac{|\ff\setminus \ff'|}{|B|}\ge \frac{|\ff|}{2q}.$$
  At the same time, the average degree of $\ff$ is $\frac kn|\ff|< \frac{|\ff|}{2q}$. This contradicts the regularity of $\ff$.
\end{proof}

\subsection{Spread approximation procedure} \label{sec31}

One common key feature of Theorem~\ref{thm2} and the following spread approximation results is the iterative procedure to construct the family $\s$. Let us abstract it in this subsection.
In what follows, we use the following notation: given two families of sets $\aaa,\s\subset 2^{[n]}$, let $\aaa[\s]$ stand for the family of all sets $F \in \aaa$ containing at least one set from $\s$.  Note that $\aaa[\s\cup \s'] = \aaa[\s]\cup \aaa[\s']$ for any three families $\aaa, \s, \s'$.

\begin{lem}\label{lem31} Fix an `ambient' family $\aaa$ and parameters $\tau, q\ge 1$. For a family $\ff\subset \aaa$ there exists  a family $\s$ of sets of size at most $q$ ({\emph a spread approximation of $\ff$}) and a `remainder' $\ff'\subset \ff$ such that
\begin{itemize}
    \item[(i)] We have $\ff\setminus \ff'\subset \aaa[\s]$;
    \item[(ii)] for any $B\in \s$ there is a family $\ff_B\subset \ff$ such that $\ff_B(B)$ is $(\aaa(B),\tau)$-homogeneous;
    \item[(iii)] $|\ff'|\le \tau^{-q-1}|\aaa|$.
  \end{itemize}
\end{lem}
The lemma is obtained using the following procedure.  For $i=1,2,\ldots $ with $\ff^1:=\ff$ we do the following steps.
\begin{enumerate}
    \item Find an inclusion-maximal set $S_i$ such that  $|\ff^i(S_i)|\ge  \tau^{|S_i|} \frac{|\aaa(S_i)|}{|\aaa|} |\ff^i|$;
    \item If $|S_i|> q$ or $\ff^i = \emptyset$ then stop. Otherwise, put $\ff^{i+1}:=\ff^i\setminus \ff^i[S_i]$.
\end{enumerate}

\begin{proof}[Verifying (i)--(iii)] The family $\ff^i(S_i)$ is $(\aaa(S_i),\tau)$-homogeneous by Observation~\ref{obs3}.

Let $m$ be the step of the procedure for $\ff$ at which we stop. Put  $\s:=\{S_1,\ldots, S_{m-1}\}$. Clearly, $|S_i|\le q$ for each $i\in [m-1]$. The family $\ff_{B}$ promised in (ii) is defined to be $\ff^i[S_i]$  for $B=S_i$. Next, note that if $\ff^m$ is non-empty, then $$|\ff^m|\le \tau^{-|S_m|} \frac{|\aaa|}{|\aaa(S_m)|} |\ff^{m}(S_m)|\le  \tau^{-|S_m|} |\aaa|.$$
We put $\ff':=\ff^m$. Since either $|S_m|>q$ or $\ff' = \emptyset $, we have $|\ff'|\le \tau^{-q-1}|\aaa|$.
\end{proof}

Note that we did not mention any potential properties of $\ff$ or its approximation $\s$. We have seen before that (ii) played a crucial role in establishing that $\s$ is $t$-intersecting for a $t$-intersecting $\ff$. In fact, for sufficiently homogeneous families, we will be able to derive even stronger properties.  We will see this in Section~\ref{sec4}.

\subsection{Spread approximations for subfamilies of spread families}
The following result is a vast generalization of Theorem~\ref{thm2} from subfamilies of ${[n]\choose k}$ to subfamilies of sufficiently spread families. As in the beginning of Section~\ref{sec3}, we present the result for the property of being  $t$-intersecting. This can be extended to several families (and more complicated properties), as we will show in the following sections.

Given a family $\aaa\subset 2^{[n]}$ of sets and $q,r\ge 1$, we say that $\aaa$ is {\it $(r,q)$-spread} if for each $S\in {[n]\choose \le q}$, the family $\aaa(S)$ is $r$-spread. Note that putting $S = \emptyset$ implies that $\aaa$ is $r$-spread, so $(r, q)$-spreadness is a stronger condition than the usual $r$-spreadness.

\begin{thm}\label{thm3}
  Let $n,k,t\ge1$ be some integers and $\aaa\subset 2^{[n]}$ be a family. Consider a family $\ff\subset \aaa\cap {[n]\choose \le k}$ that is  $t$-intersecting. Let $q, \tau \ge 1$ and assume that $\aaa$  is $(r_0,q)$-spread with $r_0$ satisfying the inequalities $r_0\ge 2\tau q$ and $r_0> 2^{12}\tau \log_2(2k)$.

  Then, in terms of Lemma~\ref{lem31}, the families $\s$ and $\ff'$ satisfy all properties guaranteed in Lemma~\ref{lem31} and, moreover, $\s$ is $t$-intersecting.
\end{thm}

We will deduce the Ahlswede--Khachatrian theorem for permutations from this theorem. We discuss more  examples of $\aaa$  to which Theorem~\ref{thm3} (and~\ref{thmapproxtrivial}) is applicable in Section~\ref{sec5}.
\begin{proof}
The only thing we need to verify is the $t$-intersection property. This proof is analogous to the proof of  Lemma~\ref{lemtint}. 

Take any (not necessarily distinct) $S_1,S_2\in \mathcal S$  and assume that $|S_1\cap S_2|<t$. Recall that for both $\ell\in \{1,2\}$ the family  $\g_\ell:=\ff_{S_\ell}(S_\ell)$ is $(\aaa(S_\ell),\tau)$-homogeneous. We use this in the second inequality below.
  $$|\g_1(\bar S_2)| \ge |\g_1|-\sum_{x\in S_2\setminus S_1} |\g_1(\{x\})|\ge \Big(1-\frac {|S_2|\tau |\aaa(S_1\cup\{x\})|}{|\aaa(S_1)|}\Big) |\g_1|\ge \frac 12|\g_1|.$$
  In the last inequality we used that $\aaa(S_1)$ is $(2\tau q)$-spread and that $|S_2|\le q$. The same is valid for $\g_2(\bar S_1)$. Note that both $\g_1':=\g_1(\bar S_2)$ and $\g_2':=\g_2(\bar S_1)$ are subfamilies of $2^{[n]\setminus (S_1\cup S_2)}.$ Moreover,  $\g_\ell'$ is $(\aaa(S_\ell),2\tau)$-homogeneous for each $\ell \in \{1,2\}$.  Indeed, this holds because of the displayed inequality and the   trivial inclusion $\g_\ell'(Y)\subset \g_\ell(Y)$, valid for any $Y$. This homogeneity implies that for any $Y\subset [n]\setminus (S_1\cup S_2)$ we have
  $$\frac{|\g_\ell'(Y)|}{|\g_\ell'|}\le (2\tau)^{|Y|}\frac{|\aaa(S_\ell\cup Y)|}{|\aaa(S_\ell)|}< \big(2^{11}\log_2(2k)\big)^{-|Y|},$$
  where the last inequality is due to the fact that $\aaa(S_\ell)$ is $r_0$-spread for $r_0>2^{12}\tau \log_2(2k)$. Thus,  $\g_\ell'$ is $r$-spread for some  $r>2^{11}\log_2(2k)$.

  The rest of the proof is identical to that of Theorem~\ref{thm2}. We are about to apply Theorem~\ref{thmtao}. Let us put $m= \log_2(2k)$ and $\delta = (2\log_2(2k))^{-1}$. Note that $m\delta = \frac 12$ and $r\delta > 2^{10}$ by our choice of $r$.  Theorem~\ref{thmtao} implies that a $\frac{1}{2}$-random subset $W$ of $[n]\setminus (S_i\cup S_j)$ contains a set from $\g_j'$ with probability strictly bigger than
  $$1-\Big(\frac 5{\log_2 2^{10}}\Big)^{\log_2 2k} k = 1-2^{-\log_2 2k} k = \frac 12.$$

  Consider a random partition of $[n]\setminus (S_1\cup S_2)$ into $2$ parts $U_1,U_2$, including each element with probability $1/2$ in each of the parts. Then both $U_\ell$, $\ell\in \{1,2\}$, are distributed as $W$ above. Thus, the probability that there is $F_\ell \in \g_\ell'$, such that $F_\ell\subset U_\ell$, is strictly bigger than $\frac 12$. Using the union bound, we conclude that, with positive probability, it holds that there are such $F_\ell$, $F_\ell\subset U_\ell,$ for each  $\ell \in\{1,2\}$. Fix such a choice of $U_\ell$ and $F_\ell$, $\ell \in \{1,2\}$. But then, on the one hand, both $F_1\cup S_1$ and $F_2\cup S_2$ belong to $\ff$ and, on the other hand, $|(F_1\cup S_1)\cap (F_2\cup S_2)| = |S_1\cap S_2|<t$, a contradiction with $\ff$ being $t$-intersecting. 
\end{proof}

\subsection{Spread approximations are simple for large families}
In a typical extremal problem, we first apply (an analogue of) Theorem~\ref{thm2} in order to get an approximation for the extremal family. The second step is to show that the approximating family is `simple'. We obtain such a result in the present subsection.

 Recall that a $t$-intersecting family $\s$ is {\it non-trivial } if $|\cap_{F\in \s} F|<t.$

\begin{thm}\label{thmapproxtrivial}
Let $\varepsilon\in (0,1]$, $n,r,q, t \ge 1$  be such that $\varepsilon r\ge 2^{17} q \log_2 q$.
Let $\aaa \subset 2^{[n]}$ be an $(r, t)$-spread family and let $\s \subset {[n] \choose \le q}$ be a non-trivial $t$-intersecting family.
Then there exists a $t$-element set $T$ such that $|\aaa[\s]| \le \varepsilon |\aaa[T]|$.
\end{thm}

For the proof, we will need the following simple observation.

\begin{obs}\label{obs22}
For any positive integers $n,p$, a family $\aaa\subset 2^{[n]}$ and a $t$-intersecting family $\s \subset {[n] \choose \le p}$ there exists a $t$-intersecting family $\T \subset {[n] \choose \le p}$ such that $\aaa[\s] \subset \aaa[\T]$ and for any $T \in \T$ and any proper subset $X \subsetneq T$ there exists $T' \in \T$ such that $|X \cap T'| < t$.
\end{obs}
One natural way to choose such $\T$ is to repeatedly replace sets in $\s$ by their proper subsets while preserving the $t$-intersecting property.

In terms of Theorem~\ref{thmapproxtrivial}, let us iteratively define the following series of families.
\begin{enumerate}
    \item Let $\T_0$ be a family given by Observation~\ref{obs22} when applied to $\aaa$ and $\s$ with $p= q$.
    \item For $i = 0, \ldots, q-t$ we put $\W_i = \T_i \cap {[n] \choose q-i}$ and let $\T_{i+1}$ be the family given by Observation~\ref{obs22} when applied to the families $\aaa$ (playing the role of $\aaa$) and  $\T_{i}\setminus \W_{i}$ playing the role of $\mathcal S$ with $p = q-i-1$.
\end{enumerate}
Remark that $\T_i$ is $t$-intersecting for each $i=0,\ldots, q-t$ by definition. We summarize the properties of these series of families in the following lemma.

\begin{lem}\label{lemkeyred} The following properties hold for each $i = 0,\ldots, q-t$. 
\begin{itemize}
  \item[(i)] All sets  in $\T_i$ have size at most $q-i$.
  \item[(ii)] We have $\aaa[\T_{i-1}]\subset \aaa[\T_{i}]\cup \aaa[\W_{i-1}]$.
  \item[(iii)] The family $\T_i$ does not have a sunflower with $q-i-t+2$ petals.
  \item[(iv)] We have $|\W_i|\le (C_0 q\log_2 q)^{q-i-t}$ with some absolute constant $C_0<2^{15}$.
  \item[(v)] If $\T_i$ consists of a single $t$-element set $X$ and this is not the case for $\T_{i-1}$ then $|\aaa[\T_{i-1}\setminus \W_{i-1}]|\le \frac{q}r |\aaa[X]|$.
\end{itemize}
\end{lem}
\begin{proof}
(i) This easily follows by induction on $i$ from the fact that all sets in $\s$ have size at most $q$ and the definition of $\T_i$.

(ii) We have $\aaa[\T_{i-1}] = \aaa[\T_{i-1} \setminus \W_{i-1}] \cup \aaa[\W_{i-1}]$ and, by the definition of $\T_{i}$, we have $\aaa[\T_i]\supset \aaa[\T_{i-1}\setminus \W_{i-1}]$.

(iii) Assume there is a sunflower $T_1,\ldots, T_{q-i-t+2}\in \T_i$ with core $X$. Assume that a set $T'\in \T_i$ intersects $X$ in $t-j$ elements, $j>0$. Then $T'$ intersects each $T_\ell$ in at least $j$ elements, implying $|T'|\ge t-j+(q-i-t+2)j = t+(q-i-t+1)j\ge q-i+1$, a contradiction with the fact that all sets in $\T_i$ have size at most $q-i$. So any $T' \in \T_i$ must intersect $X$ in at least $t$ elements. This, however, contradicts the property of $\T_i$ guaranteed by Observation~\ref{obs22}: the set $X\subsetneq T_j$ intersects all sets from $\T_i$ in at least $t$ elements.

(iv) This is trivial for $i = q-t$ since $\T_{q-t}$ contains at most $1$ set. In what follows, we assume that $i<q-t$. Take any set $Y\in \W_i$. Since $\T_i$ is $t$-intersecting, there is a $t$-element subset $X\subset Y$ such that $|\W_i| \le {q-i \choose t} |\W_i(X)|={q-i \choose q-i-t} |\W_i(X)|$. Next, $\W_i(X)$ is $(q-i-t)$-uniform and does not contain sunflowers with $(q-i-t+2)$ petals by (iii). From \eqref{eqalweiss} we conclude that, for an absolute constant $C = 2^{10}$,
\begin{align*}
|\W_i| \le& {q-i \choose q-i-t} |\W_i(X)| \le {q-i \choose q-i-t}\big(C(q-i-t+2)\log_2\big((q-i-t+2)(q-i-t)\big)\big)^{q-i-t}\\
\le& \Big(\frac{e(q-i)}{q-i-t}\Big)^{q-i-t}\big(6C(q-i-t)\log_2 q\big)^{q-i-t}\le (20C q\log_2 q)^{q-i-t}.
\end{align*}

(v) Let us assume that $\T_i = \{X\}$ for some $t$-element set $X$. Note that all sets in $\T_{i-1}$ have size at least $t+1$. Otherwise, if there is $T\in \T_{i-1}$ of size $t$ then $T$ is a proper subset of all other sets from $\T_{i-1},$ which contradicts the property of $\T_{i-1}$ guaranteed by Observation~\ref{obs22}. Thus, the sets in $\T':=\T_{i-1}\setminus \W_{i-1}$, if any, have size at least $t+1$ and all contain $X$. Recall that, for a family $\ff$, $\tau(\ff)$ is the size of the smallest set $Y$ such that $Y\cap F\ne \emptyset$ for each $F\in \ff.$ Assume that $\tau(\T'(X))> q$. Each set in $\W_{i-1}$ either contains $X$ or intersects every set from $\T'(X)$. In the latter case, it has size at least $\tau(\T'(X))$, which is impossible because each set in $W_{i-1}$ has size at most $q$. Thus, all sets from $\W_{i-1}$ contain $X$, implying that all sets from $T_{i-1}$ contain $X$, a contradiction. Therefore, $\tau(\T'(X))\le q$. If $\{x_1, \ldots, x_q\}$ is a covering of $\T'(X)$ then we have
$$
|\aaa[\T']| \le |\aaa[X \cup \{x_1\}]| + \ldots + |\aaa[X\cup \{x_q\}]| \le \frac qr |\aaa[X]|.
$$
\end{proof}

\begin{proof}[Proof of Theorem~\ref{thmapproxtrivial}] 
Fix $i$ as in Lemma~\ref{lemkeyred} (v). Note that by (i) such a choice always exists. Let $T$ be a $t$-element set such that $|\aaa[T]|$ is the largest possible. By $(r, t)$-spreadness, for any $j < i$ and any $W \in \W_j$ we have $|\aaa[W]| \le r^{-(q-j - t)}|\aaa[T']| \le r^{-(q-j - t)} |\aaa[T]|$, where $T' \subset W$ is an arbitrary subset of size $t$. By (iv) and a union bound, we get $|\aaa[\W_j]| \le r^{-(q-j-t)} (C_0 q \log_2 q)^{(q-j-t)} |\aaa[T]|$. Using this and (v) we obtain
\begin{align*}
|\aaa[\s]|\overset{(ii)}{\le} |\aaa[\T_{i-1}]|+\sum_{j=0}^{i-1}|\aaa[\W_j]|\overset{(iv),(v)}{\le}& \Big(\frac q r +\sum_{j=1}^{\infty} r^{-j}( C_0 q \log_2 q)^{j} \Big)|\aaa[T]|\\ \le  \ \ &\Big(\frac \epsilon 2 +\sum_{j=1}^{\infty} \big(\frac \epsilon4\big)^j \Big)|\aaa[T]|\le \varepsilon |\aaa[T]|,
\end{align*}
where in the third inequality we used the condition on $r$ and the bound on $C_0$.
\end{proof}

\subsection{Application to permutations: proof of part 1 of Theorem~\ref{thmintro2}}\label{sec34}
Let $\Sigma_n$ be the set of permutations on $n$ elements. We view $\Sigma_n$ as an $n$-uniform subfamily in ${[n]^2 \choose n}$ by identifying a permutation $\sigma:[n]\to [n]$ with the set $\{(1,\sigma(1)),\ldots, (n,\sigma(n))\}\in {[n]^2\choose n}$. From now on, we adopt the set terminology to permutations. We say that $S\subset [n]^2$ is a {\it partial permutation} if $S\subset \sigma$ for a permutation $\sigma\in\Sigma_n$. For any partial permutation $S$ we have $|\Sigma_n(S)| = (n-|S|)!$. If $S$ is not a partial permutation then $|\Sigma_n(S)| = 0$. For any partial permutations $S,X$, such that $S\subset X$, we have $\frac{|\Sigma_n(X)|}{|\Sigma_n(S)|} = \frac {(n-|X|)!}{(n-|S|)!}$, and simple calculus shows that for any $|S|<n/4$ and any $X\supset S$, where $X,S$ are partial permutations, we have
$$\frac{|\Sigma_n(S)|}{|\Sigma_n(X)|} = \frac {(n-|S|)!}{(n-|X|)!}\ge \big((n-|S|)!\big)^{\frac{|X|-|S|}{n-|S|}}\ge \Big(\frac{n-|S|}{e}\Big)^{|X|-|S|}>\Big(\frac n4\Big)^{|X|-|S|}.$$
If $X$ (or even $S$) is not a partial permutation then we trivially have $|\Sigma_n(S)|\ge n^{|X|-|S|}|\Sigma_n(X)|$. That is, $\Sigma_n$ is $(\frac n4,\frac n4)$-spread.

Next, we apply Theorem~\ref{thm3} to $\Sigma_n$ with $\tau = 2,$ $r_0 = \frac n4$ and $q = 4t \log_2 n$. Using the spreadness of $\Sigma_n$, we see that all the conditions of Theorem~\ref{thm2} are satisfied, provided $\frac n4\ge 4q,$ $q\le n/4$, and $\frac n4>2^{13} \log_2(2n).$ The latter is satisfied for any $n>2^{18}$, and the former two are true for any $t\le \frac n{64 \log_2 n}$. We get the following result.

\begin{thm}\label{thmpermapprox}
Let $n,t\ge 1$ be some integers satisfying $n>2^{18}$ and $t\le \frac n{64 \log_2 n}$. Consider a family $\ff\subset \Sigma_n$ of permutations that is $t$-intersecting. Put $q = 4t \log_2 n.$ Then   there exists  a $t$-intersecting family $\s$ of  partial permutations of size at most $q$  and a `remainder' $\ff'\subset \ff$ such that
\begin{itemize}
    \item[(i)] We have $\ff\setminus \ff'\subset \Sigma_n[\s]$;
    \item[(ii)] for any $B\in \s$ there is a family $\ff_B\subset \ff$ such that $\ff_B(B)$ is $(\Sigma_n(B),2)$-homogeneous;
    \item[(iii)] $|\ff'|\le n^{-4t}|\Sigma_n|$.
  \end{itemize}
\end{thm}

\begin{proof}[Proof of Theorem~\ref{thmintro2} part 1]
Let $\p\subset \Sigma_n$ be a $t$-intersecting family of permutations. We apply Theorem~\ref{thmpermapprox} to $\p$ and obtain a spread approximation family $\s$ of partial permutations (with $\p = \ff,$ $\p' = \ff'$) that is $t$-intersecting. Assume that $\s$ is non-trivial. If $\frac n4>2^{18} q\log_2 q$ and $\frac n4> t$ then we can apply Theorem~\ref{thmapproxtrivial} with $\epsilon=\frac 12$ to $\Sigma_n$ playing the role of $\aaa$ and $\s$ playing the role of $\s$ to conclude that, for some $t$-element partial permutation $T$,
$$|\p|\le |\p'|+|\Sigma_n[\s]|\le |\p'|+\frac 12 \Sigma_n[T]\le n^{-4t}n!+\frac 12(n-t)!<\frac 23(n-t)!.$$
This holds in the assumption $n>2^{20}q\log_2q$, which is implied by $n> 2^{20}q\log_2 n = 2^{22}t \log_2^2 n$.

Next, suppose that $\s$ consists of a single $t$-element partial permutation $T$.
The final step of the proof is to show that either $\p'$ is empty (and so $\p$ is a trivial $t$-intersecting family of size at most $(n-t)!$) or $|\p| \le \frac 23(n-t)!$ holds. Indeed, assume that there is a permutation $\sigma\in \p$ such that $T\not\subset \sigma$. Then $\p\setminus \p'\subset \Sigma[T]\setminus\Sigma'[T]$, where $\Sigma'[T]$ is the set of all permutations that contain $T$ and that do not share elements with $\sigma$ outside $T$. It should be clear that the number of permutations in $\Sigma'(T)$ is at least the number of derangements $[n]\setminus T\to [n]\setminus T$, which is at least $\frac{(n-t)!}e-1$. We conclude that
$$|\p|\le |\Sigma[T]\setminus \Sigma'[T]|+|\p'|\le (1-1/e)(n-t)!+1+n^{-4t} n!\le \frac 23(n-t)!.$$
This concludes the proof.
\end{proof}

\section{Spread approximations for Erd\H os--S\'os}\label{sec4}
In this section, we explore the phenomenon that can be roughly described as follows: if in the statements of Theorem~\ref{thm2},~\ref{thm3} and~\ref{thmpermapprox} and the like we take $\tau$ to be very close to $1$, then we can arrive at a spread approximation $\s$ that is $t$-intersecting while only imposing on $\ff$ that it is $(t-1)$-avoiding.

In this section, we decided to first separately spell out the proof for permutations, then give a general statement, and finally derive the statement for the classical Erd\H os--Sos for sets from the general statement.
\subsection{Erd\H os--S\'os for permutations}\label{sec41}
\begin{thm}\label{thmpermapprox2}
Let $n,t\ge 1$ be some integers satisfying $n\ge 2^{22}$ and $n\ge 2^{10} t^3\log_2 n$. Consider a family $\ff\subset \Sigma_n$ of permutations that avoids intersection $t-1$. Put $q = 32t^2\log_2 n$. Then   there exists  a $t$-intersecting family $\s$ of  partial permutations of size at most $q$  and a `remainder' $\ff'\subset \ff$ such that
\begin{itemize}
    \item[(i)] We have $\ff\setminus \ff'\subset \Sigma_n[\s]$;
    \item[(ii)] for any $B\in \s$ there is a family $\ff_B\subset \ff$ of permutations extending $B$ such that $\ff_B(B)$ is $(\Sigma_n(B),2^{1/8t})$-homogeneous;
    \item[(iii)] $|\ff'|\le n^{-4t}|\Sigma_n|$.
  \end{itemize}
\end{thm}

\begin{proof}
We apply Lemma~\ref{lem31} to $\ff$ with $\tau = 2^{1/8t}$ and $q = 32t^2\log_2 n$ and immediately get the properties (i)--(iii) listed above. The only (and crucial) thing to verify is the $t$-intersection property of $\s.$ Note that, as such, we do not even have guarantees that the sets in $\s$ have size at least $t$, but then this is of course excluded if $\s$ is indeed $t$-intersecting.

Suppose that there are sets $S_1, S_2 \in \s$ such that $|S_1 \cap S_2| < t-1$. The strategy of the proof is to find a good partial permutation $H$ such that, first, $S_1\cup H$ intersects $S_2\cup H$ in exactly $t-1$ elements and, second, the families $\ff_{S_1}(S_1\cup H)$ and $\ff_{S_2}(S_2\cup H)$  are sufficiently spread so that we could apply the same random coloring argument as, say, in the proof of Theorem~\ref{thmpermapprox}. It is in order to guarantee the existence of such $H$ that we need the approximation to be extremely homogeneous.

For a family $\mathcal W$ and $X\subset Y\subset [n]$ we denote $$\mathcal W(X,Y) = \{A\setminus X: A\in \mathcal W, A\cap Y = X\}.$$

Denote $\ell:=t-1-|S_1 \cap S_2|$ and $s_j = |S_j|,$ $j\in[2]$.
Put $\g_1 = \ff_{S_1}(S_1,S_1\cup S_2)$ and define $\g_2$ analogously. Then we have
$$
|\g_1| \ge |\ff_{S_1}(S_1)| - \sum_{i\in S_2\setminus S_1} |\ff_{S_1}(S_1\cup \{i\})| \ge \Big(1- q\tau \frac{|\Sigma_n(S_1\cup \{i\})|}{|\Sigma_n(S_1)|}\Big)|\ff_{S_1}(S_1)|  \ge 0.9 |\ff_{S_1}|,
$$
and analogously for $\g_2$. In the last inequality we used that $\frac{|\Sigma_n(S_1\cup \{i\})|}{|\Sigma_n(S_1)|}\le \frac 2n$ for $|S_1|\le \frac n2$ and that $20q\tau< 40q< n$ in our assumptions.

Let $H$ be a random $\ell$-element set taken uniformly at random from the family $\mathcal H$ of all partial $\ell$-element permutations  that satisfy $H\cap (S_1\cup S_2) = \emptyset.$ 
We clearly have $|\mathcal H| \le {n\choose \ell} n!/(n-\ell)!$. On the other hand, for any $F \in \g_j$ there are at least ${n - 2q \choose \ell}$ partial permutations $H \in \mathcal H$ such that $H \subset F$, and so for fixed $F$ we have\footnote{We omit simple calculations below}
$$
\Prb [H \subset F] \ge {n - 2q \choose \ell} /|\mathcal H| \ge {n - 2q \choose \ell}{n\choose \ell}^{-1}\frac {(n-\ell)!}{n!}\ge 0.7 \frac {(n-\ell-s_j)!}{(n-s_j)!}.
$$

The last inequality is due to the fact that $n\ge 32qt$ and $\ell<t.$ (To get a rough idea why the last inequality holds, note that both second to last and last expressions in the displayed formula have order $n^{-\ell}$.) In what follows, we denote $\gamma_j =  \frac {(n-\ell-s_j)!}{(n-s_j)!}$ for shorthand.

Combining the last two displayed inequalities, we get that
$$\E |\g_j(H)| \ge 0.7 \gamma_j|\g_j| \ge 0.6 \gamma_j|\ff_{S_j}|.$$
On the other hand, using homogeneity of $\ff_{S_j}(H),$ for any $H \in \mathcal H$ we have
$$
|\g_j(H)| \le |\ff_{S_j}(H)| \le \tau^{\ell} \gamma_j|\ff_{S_j}| < 1.1 \gamma_j |\ff_{S_j}|.
$$
Let $p_j$ be the probability that $|\g_j(H)| < 0.1 \gamma_j|\ff_{S_j}|$. We clearly have
$$
\E |\g_j(H)| < p_j \cdot 0.1 \gamma_j |\ff_{S_j}| + (1-p_j) \cdot 1.1 \gamma_j|\ff_{S_j}|.
$$
Substituting the lower bound for $\E |\g_j(H)|$, the inequality above implies the following inequality:
$$
0.6 < 0.1 p_j + 1.1 (1- p_j).
$$
From here, we get that $p_j < 0.5$. Using union bound, there exists $H \in \mathcal H$ such that $|\g_j(H)| \ge 0.1 \gamma_j|\ff_{S_j}|$ for $j = 1$ and $2$ simultaneously.

Denote $\mathcal B_j = \g_j(H)$ ($\subset \ff_{S_j}(H \cup S_j)$), then for any set $X$ of size $x$ and disjoint from $H \cup S_j$ we have
\begin{multline*}
|\mathcal B_j(X)| \le |\ff_{S_j}(S_j\cup H\cup X)|\le \tau^{\ell+x} \frac{(n-\ell-x-s_j)!}{(n-s_j)!}|\ff_{S_j}| \\ \le  \tau^{\ell+x} \frac{(n-\ell-x-s_j)!}{(n-s_j)!}10 \gamma_j^{-1}|\mathcal B_j| =  10 \tau^{\ell+x} \frac{(n-\ell-x-s_j)!}{(n-\ell-s_j)!}|\mathcal B_j|.
    \end{multline*}
Estimating $10 \tau^{\ell +x} \le (10 \tau^{\ell + 1})^x$ and applying a simple inequality $(a-b)! \le (a/e)^{-b} a!$ we get that the family $\mathcal B_j$ is $r$-spread for $r>\frac {n-\ell-s_j}{40\tau^{\ell+1}}>\frac n{80}$. 
Provided that $\frac n{80}\ge2^{11}\log_2(2n)$, which is valid for $n\ge 2^{22}$, we can employ the same argument as in several proofs before. We randomly color the set $[n]^2\setminus (S_1\cup S_2\cup H)$ and, using Theorem~\ref{thmtao}, get that there is $B_j \in \mathcal B_j$, $j = 1, 2$, such that $B_1\cap B_2 = \emptyset$. But then the permutations $F_j = B_j \sqcup S_j \sqcup H$ belong to $\ff$ and intersect in exactly $t-1$ elements. This is a contradiction, and we conclude that $\s$ is $t$-intersecting.
\end{proof}

\subsection{Proof of Theorem~\ref{thmintro2} part 2}
Let $\p\subset \Sigma_n$ be the largest family of permutations that avoids intersection $t-1$. We apply Theorem~\ref{thmpermapprox2} to $\p$ and obtain a spread approximation family $\s$ of partial permutations (with $\p = \ff,$ $\p' = \ff'$) that is $t$-intersecting.  Assume that $\s$ is non-trivial. If $\frac n4>2^{18} q\log_2 q$ and $\frac n4> t$ then we can apply Theorem~\ref{thmapproxtrivial} with $\epsilon = 1/2$ to $\Sigma_n$ playing the role of $\aaa$ and $\s$ playing the role of $\s$ to conclude that, for any $t$-element partial permutation $T$,
$$|\p|\le |\p'|+|\Sigma_n(\s)|\le |\p'|+\frac 12 |\Sigma_n(T)|\le n^{-4t}n!+\frac 12(n-t)!<\frac 23(n-t)!.$$
That is, $\p$ is not the largest $t$-intersecting family of permutations, a contradiction. Consequently, $\s$ consists of a single $t$-element partial permutation $T$, provided $n>2^{20}q\log_2q$, which is implied by $n\ge 2^{20}q\log_2 n = 2^{25}t^2 \log_2^2 n$. To apply Theorem~\ref{thmpermapprox2}, we needed $n\ge 2^{22},$ which is implied by the previous inequality, and $n\ge 2^{10}t^3\log_2 n.$

The final step of the proof is to show that $\p'$ is empty. Reordering the elements, we may assume that $T$ is $Id_{[t]}$, i.e., a partial permutation $\{(1,1),(2,2),\ldots,(t,t)\}$.   Assume that there is a permutation $\pi\in \p$ such that $\pi\not\supset Id_{[t]}$. We analyze this situation in the following proposition.
\begin{prop}\label{55}
Let $1 \le t \le n/4$ and $n\ge 10$.
Let $\pi \in \Sigma_n$ be a permutation such that $\pi\not \supset  Id_{[t]}$. Let us put  $\g_i:=\{\sigma \in \Sigma_n:\sigma\supset Id_{[t]}, \ |\sigma\cap \pi| = i\}$. Then $\g_{t-1}$ has size  at least $\frac 14 t^{-t} (n-t)!$.
\end{prop}

\begin{proof} In this proof, it is again convenient for us to identify permutations and the corresponding sets in ${[n]^2\choose n}.$ Let us put 
$m = t-|[t]^2\cap \pi|$, $m'= t-|Id_{[t]}\cap \pi|$ 
and define a partial permutation $W = [t+1,n]^2\cap \pi$. Note that $m'\ge m$ and $|W|= n-t+m$. For a partial permutation $S\subset W$,  
let $\mathcal A_S \subset \Sigma_n$ be the set of permutations $\sigma$ such that $Id_{[t]}, S\subset \sigma.$ 
Clearly, $|\mathcal A_S| = (n-t-|S|)!$. For a given set $S$ of size $s,$ $s\le t$, and any $\sigma \in \mathcal A_S$ we have $|\sigma\cap \pi| \ge t-m' + s$, and in order to have equality, $\sigma$ and $\pi$ should share no common elements outside $S$ and $[t]^2$. Let $\mathcal A'_S\subset \mathcal A_S$ be the family of permutations $\sigma \in \mathcal A_S$ that satisfy $|\sigma\cap \pi| = t-m'+s$. The number of permutations in $\mathcal A'_S$ is at least the number of derangements of a set of size $n-t-s$, which is at least $\frac{(n-t-s)!}e-1\ge \frac 14(n-t-s)!$ by our assumptions on $n,t,s$. Moreover, $\mathcal A'_S$ are pairwise disjoint for different $s$-element $S$ and $\mathcal A'_S\subset \g_{t-m'+s}$. Therefore,
\begin{align*}|\g_{t-m'+s}|\ge \sum_{S\in {W\choose s}}|\aaa'_S| \ge& \frac 14{n-t-m\choose s}(n-t-s)!\\ \ge&  \frac 14\frac{(n-t)!}{s!}/\prod_{i=0}^{m-1}\frac{n-t-i}{n-t-s-i}\\ \ge& \frac 14\frac{(n-t)!}{s!} e^{-\frac{sm}{n-t-s-m}}\\
\ge& \frac 14\frac{(n-t)!}{s!e^{s}}\ge \frac 14\frac{(n-t)!}{(s+1)^{s+1}}, \end{align*}
where in the second to last inequality we used $m,s\le t$, implying $n-t-s-m\ge m,$ and in the last inequality we used that $s!e^s\le (s+1)^{s+1}$ for any positive integer $s$ (it is easy to verify by induction). Substituting $s = m'-1$ and using $m'\le t$ again, we get that $|\g_{t-m'+s}|\ge \frac 14 t^{-t}(n-t)!$. \end{proof}


It should be clear that  $\p\setminus \p'\subset \Sigma_n[T]\setminus\g_{t-1}$. The proposition implies that
$$|\p|\le |\Sigma_n[T]|- |\g_{t-1}|+|\p'|\le (n-t)!-\frac 14 t^{-t} (n-t)!+n^{-4t} n!< (n-t)!.$$
This concludes the proof.

\subsection{Erd{\H o}s--S\'os for general families}

In this section, we are going to generalize the argument from Section \ref{sec41} from permutations to arbitrary families. To make it work, we need to impose an additional `regularity' condition on the families in question.

Given a family $\aaa \subset 2^{[n]}$ and $t \ge 0$, the {\it $t$-th shadow of $\aaa$} is the family $\hh_t \aaa \subset {[n] \choose t}$ defined as follows:
$$\hh_t \aaa = \bigcup_{X\in\aaa} {X\choose t}.$$
Denote $\hh_{\le t} \aaa = \bigcup_{i = 0}^t \hh_i \aaa$.
We say that a $k$-uniform family $\aaa \subset {[n] \choose k}$ is {\it $(t, q, \varepsilon, \theta)$-regular} if for any $\ell \le t$ and for any set $S \in \hh_{\le q}\aaa$ the following two conditions are satisfied:

\begin{itemize}
    \item [(i)] We have $|\hh_\ell(\aaa(S))| \ge (1-\varepsilon) |\hh_\ell(\aaa)|$.
    \item[(ii)] For a uniformly random $\underline{H}$ in  $ \hh_\ell(\aaa(S))$ we have
    \begin{equation}\label{aaaa}
        \Prb\big(|\aaa(S \cup \underline{H})| \ge \theta \E|\aaa(S \cup \underline{H})|\big)\ge 1-\varepsilon.\end{equation}
\end{itemize}

\begin{thm}\label{genES}
Let $n \ge k \ge t \ge 1$. Let $\varepsilon \in (0, 0.01]$, $\theta \in (0, 1]$, $r > 1$ and $q \ge t$ satisfy the following inequality:
$$
r \ge \max \left\{2^{15} \log k, 2^{17} q \log q, q/(\varepsilon \theta)\right\}.
$$
Let $\aaa \subset {[n] \choose k}$ be a $(t, q,\varepsilon, \theta)$-regular and $(r, 2q)$-spread family and let $\ff \subset \aaa$ be a subfamily such that $|F \cap F'| \neq t-1$ for any $F, F' \in \ff$. Then there exists a $t$-element set $T \subset [n]$ such that either
$$
|\ff \setminus \aaa[T]| \le e^{- \frac{ \varepsilon \theta q}{t}} |\aaa|,
$$
or we can bound the size of $\ff:$ $$
|\ff| \le e^{- \frac{ \varepsilon \theta q}{t}} |\aaa| + \frac{2^{17} q \log q}{r} |\aaa(T)|.
$$

\end{thm}
The conditions in this theorem are quite technical. The payoff, however, is that many interesting `ambient' families $\aaa$ satisfy these conditions. We will give several concrete examples in Section~\ref{sec5}.

The proof relies on the following technical, but crucial, lemma which, similarly to the argument in Section \ref{sec41}, allows us to pass from a family that avoids intersection $(t-1)$ to a $t$-intersecting spread approximation.

\begin{lem}\label{ess}
Let $n \ge k \ge t \ge 1$.
Let $\varepsilon \in (0, 0.01], \theta \in (0, 1]$, $\tau \ge 1$, $r > 2^{15} \log k$ and $q \ge 1$ satisfy $
q \le \varepsilon \theta r$ and $\tau^t \le 1 + 2\varepsilon \theta$.

Let $\aaa \subset {[n] \choose k}$ be a $(t, q, \varepsilon, \theta)$-regular and $(r, 2q)$-spread family. Let $S_1, S_2$ be sets of size at most $q$ such that $|S_1 \cap S_2| \le t-1$ and for $i = 1, 2$ let $\ff_i \subset \aaa(S_i)$ be a $(\aaa(S_i), \tau)$-homogeneous family. Then there exist $F_1 \in \ff_1$ and $F_2 \in \ff_2$ such that $|(F_1 \cup S_1) \cap (F_2 \cup S_2)| = t-1$.
\end{lem}

We postpone the proof of the lemma until the next subsection and deduce the theorem from the lemma.

\begin{proof}[Proof of Theorem \ref{genES}]
Let $\ff$ be as in the statement of the theorem, denote $\tau = 1 + 1.5\varepsilon \theta t^{-1}$ and apply Lemma \ref{lem31} to the subfamily $\ff \subset \aaa$ with parameters $\tau, q$. We obtain some family $\s \subset {[n] \choose \le q}$ and subfamilies $\ff'$ and $\ff_B \subset \ff, B\in \s,$ satisfying the conditions (i)-(iii) from Lemma \ref{lem31}. Lemma~\ref{lem31} gives
$$
|\ff'| \le \tau^{-q-1} |\aaa| \le e^{-\frac{\varepsilon \theta q}{t}} |\aaa|.
$$
The family $\aaa$ satisfies the conditions of Lemma \ref{ess}, moreover, it is easy to check that the inequalities on the parameters are also satisfied. Applying the lemma to each pair of sets $S_1, S_2 \in \s$ and families $\ff_{S_1}(S_1) \subset \aaa(S_1)$, $\ff_{S_2}(S_2) \subset \aaa(S_2)$, we conclude that $|S_1 \cap S_2| \ge t$ and thus that $\s$ is a $t$-intersecting family.

If $\s$ is a trivial $t$-intersecting family, i.e., if there is a $t$-element set $T$ such that $T\subset B$ for any $B\in \s$, then $\ff \setminus \ff' \subset \aaa[T]$ and so
$$
|\ff \setminus \aaa[T]| \le |\ff'| \le e^{-\frac{\varepsilon \theta q}{t}} |\aaa|.
$$
If $\s$ is a non-trivial $t$-intersecting family then we can apply Theorem \ref{thmapproxtrivial} with $\varepsilon = \frac{2^{17} q \log q}{r}$, and conclude that there exists $T \in {[n] \choose t}$ such that $|\ff \setminus \ff'| \le |\aaa[\s]| \le \frac{2^{17} q \log q}{r} |\aaa[T]|$. So we obtain
$$
|\ff| \le |\ff'| + |\ff \setminus \ff'| \le e^{-\frac{\varepsilon \theta q}{t}} |\aaa| + \frac{2^{17} q \log q}{r} |\aaa(T)|.
$$
Both cases of Theorem \ref{genES} are covered.
\end{proof}

\subsection{Proof of Lemma \ref{ess}}

In the setting of Lemma \ref{ess}, denote $\ell = t - 1 - |S_1 \cap S_2|$.
For $i = 1,2$ put $\ff_i':=\ff_i(\overline{S_{3-i}\setminus S_{i}})$. Note that any set $F \in \ff_1'$ satisfies $F\cap (S_1\cup S_2) = \emptyset$. 
We have
\begin{align*}
    |\ff_i'| \ge& |\ff_i| - \sum_{x\in S_{3-i}\setminus S_i}\frac{|\ff(S_i \cup \{x\})|}{|\ff(S_i)|} |\ff_i| \\
    \ge& |\ff_i| - \sum_{x\in S_{3-i}\setminus S_i} \tau \frac{|\aaa(S_i \cup \{x\})|}{|\aaa(S_i)|} |\ff_i| \\
    \ge& |\ff_i| - \tau \frac{|S_{3-i}\setminus S_i|}{r} |\ff_i|\ge \left(1 - \tau \frac{q}{r}\right) |\ff_i|,
\end{align*}
where we used the $(\aaa(S_i), \tau)$-homogeneity of the family $\ff_i$ in the second inequality, $r$-spreadness of $\aaa(S_i)$ in the third inequality, and $|S_{3-i}|\le q$ in the fourth inequality.

Fix $i\in [2]$ and let $\underline{H}$ be a uniformly random element of $\hh_\ell(\aaa(S_i))$. Denote $\xi_i = |\ff_i'(\underline{H})|$ and $\eta_i = \frac{|\aaa(S_i \cup \underline{H})|}{|\aaa(S_i)|} |\ff_i'|$. For any $F \in \aaa(S_i)$ the number of $H \in \hh_\ell(\aaa(S_i))$ such that $H \subset F$ is ${k-|S_i| \choose \ell}$, and so for any $\g\subset \aaa(S_i)$
$$
\E |\g(\underline{H})| =  |\g|\frac{{k - |S_i| \choose \ell}}{|\hh_\ell(\aaa(S_i))|} .
$$
Therefore,
$$
\E \eta_i = \frac{|\ff_i'| }{|\aaa(S_i)|} \E |\aaa(S_i \cup \underline{H})| = \frac{|\ff_i'| }{|\aaa(S_i)|} \frac{|\aaa(S_i)|}{|\hh_\ell(\aaa(S_i))|} {k - |S_i| \choose \ell} = \E \xi_i.
$$
Since $\ff_i$ is $(\aaa(S_i), \tau)$-homogeneous, for any $\underline{H}$ we have:
$$
\xi_i = |\ff_i'(\underline{H})| \le |\ff_i(\underline{H})|\le \tau^\ell \frac{|\aaa(S_i \cup \underline{H})|}{|\aaa(S_i)|} |\ff_i| \le \tau^\ell\left(1 - \tau \frac{q}{r} \right)^{-1}  \eta_i \le (1+4 \varepsilon \theta) \eta_i,
$$
where the last inequality uses $\tau^\ell\le 1+2\varepsilon\theta$, $q\le \varepsilon\theta r$ and some easy but somewhat tedious estimates. 
Since the distribution of $\underline{H}$ is uniform, the second condition of $(t, q, \varepsilon, \theta)$-regularity states
$$
\Prb(\eta_i \ge \theta \E\eta_i) \ge 1-\varepsilon.
$$

The difference $(1+4\varepsilon \theta) \eta_i -\xi_i$ is a non-negative random variable with expectation $4\varepsilon \theta \E\xi_i$. So by Markov's inequality,
$$
\Prb((1+4\varepsilon \theta) \eta_i -\xi_i \ge 0.5\theta \E\xi_i) \le \frac{4\varepsilon\theta}{\theta/2} = 8\varepsilon,
$$
and so we obtain
$$
\Prb(\xi_i \le 0.5 \eta_i) \le \Prb(\eta_i \le \theta \E\xi_i) + \Prb((1+4\varepsilon \theta) \eta_i -\xi_i \ge 0.5\theta \E\xi_i) \le 9\varepsilon.
$$

Thus, $\Prb(\xi_i > 0.5 \eta_i)\ge 1-9\varepsilon$. Property (i) in the definition of $(t,q,\varepsilon,\theta)$-regularity applied to $\aaa(S_i)$ imlpies that a uniformly random $\underline{H'}\in \partial_l(\aaa)$ belongs to $\partial_\ell(\aaa(S_i))$ with probability $1-\varepsilon$, and, conditioned on this event, is distributed uniformly in $\partial_\ell(\aaa(S_i))$. Unveiling the definitions of $\xi_i,\eta_i$, we get
$$\Prb\left(|\ff'_i(\underline{H'})| \ge 0.5 \frac{|\aaa(S_i \cup \underline{H'})|}{|\aaa(S)|}|\ff'_i|\right) \ge  \Prb\big(\underline{H'}\in \partial_\ell(\aaa(S_i))\big)\cdot \Prb(\xi_i\ge 0.5\eta_i)\ge 1-10\varepsilon.$$
Therefore, with probability at least $1-20\varepsilon > 0$ we have
$$|\ff'_i(\underline{H'})| \ge 0.5 \frac{|\aaa(S_i \cup \underline{H'})|}{|\aaa(S)|}|\ff'_i|$$
 for both $i=1,2$. For any $H\in \hh_\ell(\aaa)$ which belongs to this event we automatically have $H \cap (S_1 \cup S_2) = \emptyset$ and for $i=1, 2$:
\begin{equation}\label{111n}
|\ff'_i(H)| \ge 0.5 \frac{|\aaa(S_i \cup H)|}{|\aaa(S)|} |\ff'_i| \ge \frac14 \frac{|\aaa(S_i \cup H)|}{|\aaa(S)|} |\ff_i|.
\end{equation}
Since $\ff_i$ is $(\aaa(S_i), \tau)$-homogeneous, for every set $Y$ disjoint from $H$ we have
$$
|\ff'_i(H \cup Y)| \le |\ff_i(H \cup Y)| \le\tau^{|H|+|Y|} \frac{|\aaa(S_i \cup H \cup Y)|}{|\aaa(S_i)|} |\ff_i| \le r^{-|Y|} \tau^{|H|+|Y|} \frac{|\aaa(S_i \cup H)|}{|\aaa(S_i)|} |\ff_i| ,
$$
where in the third inequality we used the $(r, 2q)$-spreadness of $\aaa$ and $|S_i\cup H|\le 2q$. So by (\ref{111n}) we get
$$
|\ff'_i(H \cup Y)| \le 4 r^{-|Y|} \tau^{|H|+|Y|} |\ff'_i(H)| \le 4 \tau^{t-1} (\tau^{-1} r)^{-|Y|} |\ff'_i(H)|,
$$
and so the family $\ff_i'(H)$ is $R$-spread with $R = \frac 14 \tau^{-t} r \ge \frac r 8 > 2^{12} \log k$. By the argument as in Lemma~\ref{lemtint} we can find sets $B_i \in \ff'_i(H)$ such that $B_1 \cap B_2 = \emptyset$. Then the sets $F_i = B_i \cup H$ satisfy the conclusion of the claim.

\subsection{Proof of Theorem~\ref{thmintro1}}
Let us first derive a proposition that states that Theorem~\ref{genES} is applicable in our situation.

\begin{prop}\label{example1}
The family $\aaa = {[n]\choose k}$ satisfies the condition of Theorem~\ref{genES} with $\epsilon = 0.01,$ $\theta = 1$, $r = n/k$, and $q = 400t^2 \log n$ for $n> 2^{26}k t^2\log^2n$ and $k\ge q+t.$ In these conditions, we get that for any $(t-1)$-avoiding $\ff\subset {[n]\choose k}$ there exists a $t$-element set $T \subset [n]$ such that either
$$
|\ff \setminus \aaa(T)| \le n^{- 4t} |\aaa|,
$$
or we can bound the size of $\ff:$ $$
|\ff| \le n^{- 4t} |\aaa| + \frac{2^{26}k t^2\log^2n}{n} |\aaa(T)|.
$$
\end{prop}
\begin{proof}
Let us first verify the regularity condition. First, for $S\in {[n]\choose \le q}$ and $\ell\le t$, $q+t\le k$, we have $|\partial_\ell(\aaa(S))| = {n-|S|\choose \ell}$ and $|\partial_\ell(\aaa)| = {n\choose \ell}$. We have $$\frac{{n\choose \ell}}{{n-|S|\choose \ell}}\le \frac{{n\choose \ell}}{{n-q\choose \ell}}\le \Big(\frac{n-\ell}{n-q-\ell}\Big)^\ell\le e^{q\ell/(n-q-\ell)}<1.01,$$
provided $n\ge q+\ell+200q\ell$, which is implied by our restrictions. Second, note that for $X\subset [n]$ of fixed size we have $|\aaa(X)| = {n-|X|\choose k-|X|}$, and thus the random variable $|\aaa(S\cup \underline{H})|$ is constant. This immediately gives (ii).

Next, we verify the $(r,2q)$-spreadness condition. For any set $S\in {[n]\choose \le 2q}$ and set $X\subset [n]\setminus S$ we have
$$\frac{|\aaa(S\cup X)|}{|\aaa(S)|} = \frac{{n-|S|-|X|\choose k-|S|-|X|}}{{n-|S|\choose k-|S|}}\le \Big(\frac{k-|S|}{n-|S|}\Big)^{|X|}\le \Big(\frac kn\Big)^{|X|}.$$
The last thing is to verify the condition on $r = n/k$:
$$n\ge k\cdot\max\{2^{15}\log k, 2^{17}q\log q, 100q\}.$$
The right-hand side is at most
$$ k\cdot\max\{2^{15}\log n, 2^{17}q\log n\} = 2^{17}kq\log n = 2^{17}k\cdot 400t^2\log^2 n.$$ A stronger inequality $n\ge 2^{26} kt^2\log^2n$ holds in our assumptions. 
\end{proof}

\begin{proof}[Proof of Theorem~\ref{thmintro1}]
We apply Proposition~\ref{example1} to a $(t-1)$-avoiding $\ff\subset \aaa:={[n]\choose k}$ for $n>2^{26}kt^2\log^2n$ and get that either $|\ff|\le n^{-4t}{n-t\choose k-t}+\frac 12 {n-t\choose k-t}<{n-t\choose k-t},$ or there exists a set $T$ of size $t$ such that $|\ff\setminus \aaa(T)|<n^{-4t}{n\choose k}$. If the family $\ff$ is extremal then the latter should hold, which we suppose below.

The remaining step is to show that $\ff' = \ff\setminus \aaa(T)$ is empty.  Put $|\ff'|=\epsilon{n\choose k}$ and fix such a set $Y\subset T$ such that $|\ff'(Y,T)|\ge |\ff'(Y',T)|$ for any other $Y'\subset T$. Denote $\mathcal H:=\ff'(Y,T)$. In particular, $|\mathcal H|\ge \epsilon{n-t\choose k-|Y|}$. Take a subset $X\subset [n]\setminus T$ of size $t-1-|Y|$ such that $|\mathcal H(X)|\ge \epsilon{n-t-|X|\choose k-t+1}$. Note that this is possible since the right-hand side is simply the expected size of $\mathcal H(X)$ for a randomly chosen $X\subset [n]\setminus T$ of size $t-1-|Y|$. Let us denote $\mathcal H_1:=\mathcal H(X)$ and note $\mathcal H_1\subset {[n]\setminus (X\cup T)\choose k-t+1}$. Choose $\delta$ such that $|\mathcal H_1| = \delta {n-t-|X|\choose k-t+1}$ and note $\delta\ge \epsilon$. Also, let us put $\ff_1:= \ff(T\cup X)$. The family $\ff_1$ is a subfamily in ${[n]\setminus (T\cup X)\choose k-t-|X|}$. Note that $\ff_1$ and $\mathcal H_1$ are cross-intersecting.  Recall that for a family $\mathcal W\subset 2^{[n]}$, the upper $\ell$-shadow $\partial^{\ell}(\mathcal W)$ of $\mathcal W$  is defined as follows: $\partial^{\ell}(\mathcal W) = \{S\in {[n]\choose \ell}: A\subset S \text{ for some }A\in \mathcal W\}.$
The families $\mathcal H_2:=\partial^{(n+k)/2}(\mathcal H_1)$ and $\ff_1$ are cross-intersecting as well. A numerical consequence of the Kruskal--Katona theorem (see \cite[Theorem 2]{BT}) implies that $|\mathcal H_2| \ge \delta^{1/2}{n-t-|X|\choose (n+k)/2}$. Let us recall the following standard fact about independent sets in regular bipartite graphs:
\begin{lem}
Let $G = (V \cup U, E)$ be a nonempty biregular bipartite graph and let $I \subset V \cup U$ be an independent set. Then we have $\frac{|I \cap V|}{|V|} + \frac{|I \cap U|}{|U|} \le 1$.
\end{lem}
To apply the lemma, we put $V = {[n] \setminus (T\cup X)\choose k-t-|X|}$, $U = {[n] \setminus (T\cup X)\choose (n+k)/2}$ and draw an edge between two sets if and only if they are disjoint. Then $I = \ff_1 \cup \mathcal H_2$ is an independent set in $G$ and so we get $|\ff_1| \le (1-\delta^{1/2}){n-t-|X|\choose k-t-|X|}$.

 We claim that $|\ff_1|+|\ff'|<{n-t-|X|\choose k-t-|X|}$ unless $\epsilon = 0$.  Indeed, $\epsilon^{1/2}\le n^{-2t}$, and thus if $\epsilon\ne 0$ then $$\delta^{1/2}{n-t-|X|\choose k-t-|X|}\ge \epsilon^{1/2}{n-2t\choose k-2t}>\epsilon {n\choose k}=|\ff'|.$$
But this implies that $|\ff[T\cup X]|+ |\ff'|<|\aaa[T\cup X]|,$ and thus $|\ff|< |\aaa[T]|$. Therefore, $\ff$ is not extremal unless $\ff' = \emptyset.$ 

Finally, we note that for $n = k^\alpha, t\le k^\beta$ with $\beta<\frac 12$ and $\alpha>1+2\beta$ both inequalities $n>2^{26}kt^2\log^2n$ and on $k>t+400t^2\log n$ are satisfied, provided $k$ is sufficiently large.
\end{proof}

\section{More examples of families for spread approximations}\label{sec5}
In this section, we give several examples of families satisfying the conditions of Theorems~\ref{thm3},~\ref{thmapproxtrivial} and~\ref{genES}. We choose the values of $q$ below so that the resulting conclusions give us a meaningful and interesting approximation result. For brevity, in some cases we do not give the full derivation of the corresponding Ahlswede--Khachatrian or Erd\H os--S\'os result since it follows the same line of argument as for ${[n]\choose k}$ or $\Sigma_n$.

We saw that Theorem~\ref{genES} is applicable to the most natural family: $\aaa = {[n]\choose k}$. We also note that the family $\aaa = \Sigma_n$ of all permutations $[n]\to [n]$ satisfies the conditions of Theorem~\ref{genES}. This should not be surprising, since the proofs of Theorems for permutations followed the same logic. We have verified the $(n/4,n/4)$-spreadness of $\Sigma_n$ in Section~\ref{sec34}. As for the regularity, part (ii) is the same as that for ${[n]\choose k}$: whenever {\it non-empty}, $\Sigma_n(X)$ depends on the size of $X$ only. Part (i) of the regularity condition is actually the most restrictive, it is the reason for the inequality $t\le n^{1/3-o(1)}$ in Theorem~\ref{thmintro2} part 2. Indeed, in order to guarantee it, we need to make sure that a random partial permutation of size  $l\le t$ (i.e. a function that maps some $l$ elements of $[n]$ to some $l$ distinct elements of $[n]$) extends some fixed partial permutation of size at most $q$ with large probability. This means that neither their domains nor images intersect. We can guarantee this only if $n>Cqt$ for some large fixed $C.$

Another example to which Theorem~\ref{genES} is applicable is the family $\aaa = {[n_1]\choose k_1}\times\ldots\times{[n_w]\choose k_w}$ for certain choices of positive integers $n_i,k_i,w$.\footnote{For set systems $\mathcal F_1$, \ldots, $\mathcal F_w$ the Cartesian product $\mathcal F_1 \times \ldots \times \mathcal F_w$ is the family of all sets of the form $(F_1 \times \{1\}) \sqcup (F_2 \times \{2\}) \sqcup \ldots \sqcup (F_w \times \{w\})$ over all $F_i \in \mathcal F_i$.}
For simplicity, let us focus on the case $\aaa = {[n]\choose k}^w$. As in the case of ${[n]\choose k},$ is not difficult to see that $\aaa$ is $\frac nk$-spread and, moreover, $(\frac nk,\ell)$-spread for any $\ell <k$.

\begin{prop}\label{example21}
Let $n,k,w,t\ge 1$ be integers. The family $\aaa = {[n]\choose k}^w$ satisfies the condition of Theorem~\ref{thm3} with $\tau = 2$ and $q = 4t\log n$ for $\frac nk>\max\{2^{13} \log_2(kw), 2^{18}q\log q\}$. Further, under the same inequality on $n$ it satisfies the condition of Theorem~\ref{thmapproxtrivial} with $\epsilon = 0.5.$

Consequently, if $\ff\subset \aaa$ is $t$-intersecting and satisfies $|\ff|\ge \frac 23|\aaa[T]|$ for any $t$-element $T$ then there is a choice of such $T$ that $|\ff\setminus \aaa[T]|<n^{-4t}|\aaa|.$
\end{prop}

Under the conditions of Proposition~\ref{example21} it is not difficult to derive that if $\ff\subset \aaa$ is $t$-intersecting and extremal then $\ff = \aaa[T]$ for some $t$-element $T.$ The derivation is virturally identical to that in the proof of Theorem~\ref{thmintro1}.

For the Erd\H os--S\'os property the conditions are a bit more restrictive. We omit the proof of Proposition~\ref{example21} and only give the proof of the following proposition.

\begin{prop}\label{example22}
Let $n,k,w,t\ge 1$ be integers. The family $\aaa = {[n]\choose k}^w$ satisfies the condition of Theorem~\ref{genES} with $\epsilon = 0.01,$ $\theta = 0.5$, $q = 800t^2 \log n$ for $\frac nk> \max\{2^{15}\log(kw), 2^{17}q\log q\}$ and $k\ge 2q.$

Consequently, if $\ff\subset \aaa$ omits intersection $(t-1)$ and satisfies $|\ff|\ge \frac 23|\aaa[T]|$ for any $t$-element $T$ then there is a choice of such $T$ that $|\ff\setminus \aaa[T]|<n^{-4t}|\aaa|.$
\end{prop}
Again, under the conditions of Proposition~\ref{example22} it is not difficult to derive that if $\ff\subset \aaa$ avoids intersection $(t-1)$ and extremal then $\ff = \aaa[T]$ for some $t$-element $T.$

Note that, unlike the previous two cases, this family is not `symmetric', and so the value of $\theta$ is not $1.$ We omit some parts of the proof since they are similar to the case of $\aaa = {[n]\choose k}$ (i.e., $w= 1$).
\begin{proof}
Let us first verify the regularity condition. The first part is done as in the case of ${[n]\choose k}$, so we verify (ii) only. Let us denote the ground set by $X = X_1\sqcup X_2\sqcup\ldots\sqcup X_w$, where $X_i$ has size $n$. Let $S\subset X$ have size $x\le q$ and $H\subset X\setminus S$ have size $\ell\le t,$ and put $s_i = |S\cap X_i|,$ $\ell_i = |H\cap X_i|.$ Then we have $|\aaa(S\cup H)| = \prod_{i=1}^w{n-s_i-\ell_i\choose k-s_i-\ell_i} = \prod_{i=1}^w{n-s_i\choose k-s_i}\prod_{j=s_i}^{s_i+\ell_i-1}\frac {k-j}{n-j}$. It is not difficult to see\footnote{This requires a bit of tedious calculations which we omit} that these expressions differ by a factor of at most $2$ for  different choices of $H$ (i.e., different choices of $\ell_i$  such that $\sum \ell_i = \ell$), provided $(k-s_i)^\ell/(k-s_i-\ell)^\ell\le 2$ for each $i.$ At the same time, the fact that these expressions differ by at most $2$ implies (ii) (with $\epsilon=0$) via averaging. We have $(k-s_i)^\ell/(k-s_i-\ell)^l\le (k-q)^\ell/(k-q-\ell)^\ell\le e^{\ell^2/(k-q-\ell)}$, which is at most $2$ provided $k>q+ 4\ell^2$. Note that $q>4t^2\ge 4\ell^2,$ and thus it is sufficient to require $k\ge 2q$ in order for this inequality to hold.

The spreadness of $\aaa$ is verified as in the case of $ {[n]\choose k}$, so we omit it.
 \end{proof}

The proposition above does not cover the case $\aaa = [n]^k$, so we state it separately without a proof.
\begin{prop}\label{example3}
Let $n,k,t\ge 1$ be integers. The family $\aaa = [n]^k$ satisfies the condition of Theorem~\ref{genES} with $\epsilon = 0.01,$ $\theta = 1$, $q = 400t^2 \log n$ for $n>\max\{2^{15}\log k, 2^{17}q\log q\}$ and $k\ge q+t.$
\end{prop}

Our last example are designs. Let $\aaa$ be an {\it $(n,k,w)$-design}, i.e., a collection of sets from ${[n]\choose k}$ such that every $W\in {[n]\choose w}$ is contained in exactly $1$ set from $\aaa.$

\begin{prop}\label{example4}
Let $n,k,w,t\ge 1$ be integers. An $(n,k,w)$-design $\aaa$ satisfies the condition of Theorem~\ref{genES} with $\epsilon = 0.01,$ $\theta = 1$, $q = 400t^2 \log n$ for $n>k \big(2^{26}t^2\log^2n)^{k/(w-2q)}$ and $w> 2q.$
\end{prop}

\begin{proof}
The proof is very similar to the one for ${[n]\choose k},$ so we will only sketch it. We start with the regularity condition. Part (i) is guaranteed because $\partial_l(\aaa(S)) = {[n]\setminus S\choose l}$, provided $|S|+l\le w$. Part (ii) is guaranteed because $\aaa$ is highly subset-regular, i.e., for any $X\subset [n]$ of fixed size $|\aaa(X)|$ is fixed.

As for spreadness, it is easy to see that $|\aaa(X)| = {n-|X|\choose w-|X|}/{k-|X|\choose w-|X|}$ for $|X|\le w,$ and $|\aaa(X)|\in \{0,1\}$ for $X>w$. Thus, we get that $|\aaa(X')|\le (k/n)^{|X'|-|X|} |\aaa(X)|$ for any $X\subset X',$ $|X'|\le w.$ This, via simple calculations, implies that $\aaa$ is $(r,2q)$-spread for $r = (n/k)^{(w-2q)/k}$. Doing calculations analogous to that in the proof of Proposition~\ref{example1}, we arrive at the bound on $n$ in the statement.
\end{proof}
The dependence of $q$ on $n$ and $w$ on $q$ is undesirable in this case, since we know of the existence of designs for very large $n = n(k,w)$ only. We can guarantee  the Ahlswede--Khachatrian type result under milder restrictions. We only need from the above proof that an $(n,k,w)$-design $\aaa$ is $(r,q)$-spread for $r = (n/k)^{(w-q)/k}.$ 
\begin{prop}\label{example42}
Let $n,k,w,t\ge 1$ be integers such that $k\ge w > t$ and $n > k (2^{18} k \log k)^{\frac{k}{w-t}}$. Let $\aaa$ be any $(n,k,w)$-design, if $\ff \subset \aaa$ is $t$-intersecting then $|\ff| \le {n-t \choose w-t}/{k-t \choose w-t}$ with equality if and only if $\ff = \aaa[T]$ for some $t$ element subset $T$.
\end{prop}

This follows directly from Theorem~\ref{thmapproxtrivial} with $q = k$ and $\varepsilon = 0.5$.


\section{Concluding remarks}\label{sec6}
In this paper, we gave a powerful and versatile spread approximations  approach to get structural and extremal results for subfamilies of `sufficiently good' ambient families that avoid certain substructures. It seems that we only scratched the surface of possible applications to extremal set theory. There are several directions to explore using the methods of the present paper.

The first direction is to broaden the class of potential `ambient' families. The second  is to explore more complicated Turan-type properties, in the spirit of the paper by Keller and Lifshitz \cite{KL}. Below an example of such a statement for several families.

We say that $\ff_1,\ldots, \ff_s$ are  $t$-cross-dependent if for any $A_1\in \ff_1,\ldots, A_s\in \ff_s$ we have $|A_1\cup\ldots\cup A_s|\le \sum_{i=1}^s|A_i|-t$. In particular, $1$-cross-dependent coincides with the usual notion of cross-dependence (i.e., the absence of a cross-matching).
\begin{thm}\label{thm22}
  Let $n,k,s\ge 2$, $t\ge1$ be some integers. Consider families $\ff_1,\ldots, \ff_s\subset {[n]\choose k}$ that are $t$-cross-dependent. Let $r,q\ge 1$ satisfy the following restriction: $r> 2^{11} s\log_2(sk)$, $r\ge 2(s-1)q$, $n\ge rk$. Then there exist $t$-cross-dependent families $\mathcal S_1,\ldots, \mathcal S_s\subset {[n]\choose \le q}$ of sets of size at most $q$ and families $\ff'_1\subset \ff_1,\ldots, \ff'_s\subset \ff_s$ such that the following holds for each $j\in[s]$.
  \begin{itemize}
    \item For any $A\in \ff_j\setminus \ff'_j$ there is $B\in\s_j$ such that $B\subset A$;
    \item for any $B\in \s_j$ there is a family $\ff_B\subset \ff_j$ such that $\ff_B(B)$ is $r$-spread;
    \item $|\ff'_j|\le \Big(\frac {rk}n\Big)^{q+1}{n\choose k}$.
  \end{itemize}
  Moreover, if $\ff_1 = \ldots = \ff_s$ then $\mathcal S_1 = \ldots = \mathcal S_s$.
\end{thm}
From the imposed restrictions, this theorem gives interesting results for $n>Cks\log(ks)$ with some explicit $C$. It allows us to extend the results for $t$-intersecting families to $t$-cross-dependent families. We do not give its proof since it is virtually identical to that of Theorem~\ref{thm2}.

The third direction is to analyze the possibilities to combine the present approach with the already existing approaches, such as the delta-system method developed by Frankl and F\"uredi \cite{Fra3, Fra4, FF1, FF2}.

We finally note that the approach we used for the Erd\H os--S\'os problem does extend to more general structures, and guarantees that the spread approximation that we obtain not only forbids  the same structure, but also all structures that can be obtained from the initial one by replacing common elements of two or more sets with elements that each belong to only one set. It is possible that this can be further strengthened, but requires further research.

\section{Acknowledgements}
 This work was supported by a grant for research centers, provided by the Analytical Center for the Government of the Russian Federation in accordance with the subsidy agreement (agreement identifier 000000D730324P540002) and the agreement with the Moscow Institute of Physics and Technology dated November 1, 2021 No. 70-2021-00138.

\end{document}